\documentclass[reqno,12pt]{amsart}
\usepackage{amsmath,amsthm,amssymb,amsfonts,amscd}
\usepackage{epsfig}
\usepackage{color}
\usepackage[all]{xy}
\usepackage{tikz}
\usepackage{adjustbox}
\usepackage{hyperref}
\usepackage{amsmath}
\usepackage{mathabx}
\usepackage{hhline}
\usepackage{graphics}
\theoremstyle{plain}
\newtheorem{theorem}{Theorem}

\newtheorem{proposition}[theorem]{Proposition}
\newtheorem*{theorem*}{Theorem}
\newtheorem*{conjecture*}{Conjecture}

\theoremstyle{definition}
\newtheorem{remark}[theorem]{Remark}

\usepackage{graphics, setspace}

\pgfmathdeclarefunction{n}{0}{}

\newcommand{\breakingcomma}{%
  \begingroup\lccode`~=`,
  \lowercase{\endgroup\expandafter\def\expandafter~\expandafter{~\penalty0 }}}

\usepackage{breqn}

\usepackage{amstext} 
\usepackage{array}
\newcolumntype{L}{>{$}l<{$}}

\newcommand{\CC}{{\mathbb{C}}}

\newcommand{\PP}{{\mathbb{P}}}

\newcommand{\ZZ}{{\mathbb{Z}}}

\newcommand{\SL}{{\mathrm{SL}}}
\newcommand{\GL}{{\mathrm{GL}}}

\newcommand{\Jac}{\mathrm{Jac}}
\newcommand{\Fix}{\mathrm{Fix}}

\newcommand{\id}{\mathrm{id}}

\newcommand{\bx}{{\bf x}}

\newcommand{\hess}{\mathrm{hess}}

\newcommand{\QR}[2]{
\left.\raisebox{0.5ex}{\ensuremath{#1}}
\ensuremath{\mkern-3mu}\middle/\ensuremath{\mkern-3mu}
\raisebox{-0.5ex}{\ensuremath{#2}}\right.}

\def\B{{\mathcal B}}
\def\C{{\mathcal C}}

\def\MF{{\mathrm{MF}}}

\def\p{\partial }

\newcommand{\ccHH}{{\mathsf{HH}}}

\begin{document}
\title{Mirror partner for a Klein quartic polynomial}
\date{\today}
\author{Alexey Basalaev}
\address{A. Basalaev:\newline Faculty of Mathematics, National Research University Higher School of Economics, Usacheva str., 6, 119048 Moscow, Russian Federation, and \newline
Skolkovo Institute of Science and Technology, Nobelya str., 3, 121205 Moscow, Russian Federation}
\email{a.basalaev@skoltech.ru}

\begin{abstract}
The results of A.~Chiodo, Y.~Ruan and M.~Krawitz associate the mirror partner Calabi--Yau variety $X$ to a Landau--Ginzburg orbifold $(f,G)$ if $f$ is an invertible polynomial satisfying Calabi--Yau condition and the group $G$ is a diagonal symmetry group of $f$.
In this paper we investigate the Landau--Ginzburg orbifolds with a Klein quartic polynomial $f = x_1^3x_2 + x_2^3x_3+x_3^3x_1$ and $G$ being all possible subgroups of $\GL(3,\CC)$, preserving the polynomial~$f$ and also the pairing in its Jacobian algebra. In particular, $G$ is not necessarily abelian or diagonal. The zero--set of polynomial $f$, called Klein quartic, is a genus $3$ smooth compact Riemann surface. We show that its mirror Landau--Ginzburg orbifold is $(f,G)$ with $G$ being a $\ZZ/2\ZZ$--extension of a Klein four--group.
\end{abstract}
\maketitle


\section{Introduction}
Mirror symmetry, in one of its formulations, relates a complex variety $X$ to a so--called \textit{Landau--Ginzburg orbifold} $(f,G)$ (\cite{IV90,V89,W93}). In the latter $f$ defines a quasihomogeneous isolated singularity and $G$ is its symmetry group. In particular ${f \in \CC[x_1,\dots,x_N]}$ is a polynomial, whose partial derivatives all vanish simultaneously if and only if $x_1 = \dots = x_N=0$ and there is a set of natural numbers $d_f,d_1,\dots,d_N$, such that $f(\lambda^{d_1}x_1,\dots,\lambda^{d_N}x_N) = \lambda^{d_f}f(x_1,\dots,x_N)$ for any $\lambda \in \CC^\ast$.

The symmetry group $G$ is a subgroup of ${\GL_f := \lbrace g \in \GL(\CC^N) \ | \ f(g \cdot x) = f(x) \rbrace}$. The diagonal operator ${j_f = \text{diag}(\exp(2 \pi \sqrt{-1} \frac{d_1}{d_f}),\dots, \exp(2 \pi \sqrt{-1} \frac{d_N}{d_f}))}$ belongs to $\GL_f$, showing that $\GL_f$ is never trivial. The variety $X$ is assumed to be either quasismooth or a smooth orbifold.

Both $X$ and $(f,G)$ have some essential $\CC$--algebras associated to them. These are the cohomology ring $H^*(X)$ if $X$ is smooth or orbifold cohomology ring $H_{orb}^*(X)$ otherwise (cf. \cite{AGV08}). Note that $H^*_{orb}(X) = H^*(X)$ if $X$ is smooth. To a pair $(f,G)$ one associates the Hochschild cohomology ring $\ccHH^*(f,G)$ of the category of $G$--equivariant matrix factorizations $\MF_G(f)$. The variety $X$ is said to be mirror to the Landau--Ginzburg orbifold $(f,G)$ if there is a $\CC$--algebra homomorphism ${H_{orb}^*(X) \cong \ccHH^*(f,G)}$.

\subsection{Landau-Ginzburg -- Calabi-Yau mirror symmetry}\label{section: LG-CY MS}

Large class of mirror pairs was discovered by Chiodo,Ruan and Krawitz in \cite{CR11,K09}. In particular, let $(f,G)$ be such that 
\begin{enumerate}
 \item $f$ is the so-called invertible polynomial. Namely $f$ is a sum of exactly $N$ monomials;
 \item Calabi--Yau condition holds: $d_1+\dots+d_N = d_f$;
 \item $G$ acts diagonally on $\CC^N$ and is a subgroup of $\SL(N,\CC)$.
\end{enumerate}
Then the Berglund--H\"ubsch--Henningson dual Landau-Ginzburg orbifold $(\widetilde f, \widetilde G)$ (see \cite{BH93,BH95,Kreu94}) can be used to construct the mirror variety by
\[
    X_{\widetilde f, \widetilde G} := \left\lbrace (x_1, \dots, x_N) \in \PP(d_1,\dots,d_N) \ | \ \widetilde f(x_1,\dots,x_N) = 0 \right\rbrace \Big\slash (\widetilde G / \langle j_{\widetilde f} \rangle).
\]
Theorem~14 of \cite{CR11} together with Theorem~2.4 of \cite{K09} give that there is a $\CC$--vector space isomorphism $\tau: {H_{orb}^*(X_{\widetilde f, \widetilde G}) \cong \ccHH^*(f,G)}$. This suggests calling $X_{\widetilde f, \widetilde G}$ a mirror pair of $(f,G)$. Up to now $\tau$ is proved to be an algebra homomorphism only for the certain examples (cf. \cite{MR11,MS12,BT3}). The major complication here lies in computing the algebra structure of  $H_{orb}^*(X_{\widetilde f, \widetilde G})$ while the algebra structure of $\ccHH^*(f,G)$ was computed explicitly in \cite{BT2,BTW16,BTW17,HLL20,BI22} using the technique developed by D.Shklyarov in \cite{Sh20}.

Let's comment on the conditions above. Condition $1$ restricts us to work with the invertible polynomials only, for which the construction of Berglund--H\"ubsch--Henningson can be applied. Condition $2$ has the name \textit{Calabi--Yau condition} because it guarantees that the first Chern class of $X_{\widetilde f, \widetilde G}$ vanishes. It follows from Condition $3$ that $\langle j_{\widetilde f} \rangle \subseteq \widetilde G $ --- this is the special property of the dual group construction. 

\subsection{Admissible symmetry groups}
The algebra $H_{orb}^*(X_{\widetilde f, \widetilde G})$ is endowed with a pairing for any $f$. However one can introduce an essential pairing on $\ccHH^*(f,G)$ only if $f$ satisfies the Calabi--Yau condition (cf. \cite{PV12}). 
Generally in mirror symmetry one considers the groups $G$, such that $G \subseteq \SL(N,\CC)$. These groups are distinguished by the fact that they preserve volume form of $\CC^N$ and therefore the pairing of both algebras.

\subsection{Mirror symmetry for curves}
More general mirror symmetry results were established in the context toric degenerations and homological algebra. In particular, Ruddat in \cite{R17} and Gross-Katzarkov-Ruddat in \cite{GKR17} investigated the mirrors of the varieties of general type from the point of view of toric degenerations via the relations on the Hodge numbers.

P.Acosta investigated in \cite{A14} the vector space structure of the algebras above for the polynomials $f$ that do not satisfy the Calabi--Yau condition. His result shows that $H^\ast(X_{f,G})$ and $\ccHH^\ast(f,G)$ are of different dimension if $f$ does not satisfy the Calabi--Yau condition.

A smooth genus $g$ surface $S_g$ is a Calabi--Yau manifolds only if $g=1$. However, this is essential to ask what is a mirror Landau--Ginzburg model for $S_g$ with $g \ge 2$. 

Efimov and Seidel \cite{E12,S11} established an equivalence between the (suitably defined) triangulated categories of a genus $g \ge 2$ surface $S_g$ and Landau--Ginzburg orbifold $(f_g,G)$ with $f_g = x_1^{2g+1}+x_2^{2g+1}+x_3^{2g+1} + x_1x_2x_3$. Here the group $G$ is generated by the diagonal operator $\text{diag}(\zeta,\zeta,\zeta^{2g-1})$, such that $\zeta = \exp(\frac{2\pi  \sqrt{-1}}{2g+1})$. Without giving the details of this mirror theorem we quote the following consequence: there is a $\CC$--algebra isomorphism $H^*(S_g) \cong \ccHH^*(f_g,G)$ (cf. \cite{E12,S11,Sh20}).

The difference between the approaches of Acosta and Efimov--Seidel is that in the first case the variety $X=X_{f,G}$ is given by a Landau--Ginzburg model $(f,G)$ whereas in the seconds case the Landau--Ginzburg model is constituted by the variety $X = S_g$. Another important difference is that $H^\ast(S_g)$ is a $\ZZ/2\ZZ$--commutative algebra while Acosta works in the context of FJRW theory (c.f. \cite{FJR})  that is always commutative.

In this note we investigate the homological mirror symmetry result of Efimov and Seidel from the point of view of Berglund--H\"ubsch--Henningson duality. 


\subsection{Klein quartic}
Consider the polynomial $f = x_1^3x_2 + x_2^3x_3 + x_3^3x_1$. We have $d_1=d_2=d_3 = 1$ and $d_f=4$. It is invertible with the Berglund--H\"ubsch dual $\widetilde f = f$. Take $J := \langle j_f \rangle$. Then the variety $X_{\widetilde f, J}$  is a smooth genus $3$ curve $S_3$. Its symmetry group is isomorphic to $\mathrm{PSL}(2,7)$ --- simple nonabelian of order $168$ (cf. \cite{K79,E98}).

Denote $\SL_f := \GL_f \cap \SL(N,\CC)$. 
For any $G \subseteq \SL_f$ denote by $\widehat G$ the central extension of $G$ by $\ZZ/2\ZZ$ obtained by adding the $-\id \in \GL_f$ to $G$.

\begin{theorem}
    Let $f = x_1^3x_2 + x_2^3x_3 + x_3^3x_1$ be a Klein quartic polynomial and $S_3$ stand for the smooth genus $3$ Riemann surface. Then
    \begin{itemize}
     \item[(a)] $\dim \ccHH^*(f,G) > \dim H^\ast(S_3)$ for any $G \subseteq \SL_f$.
     \item[(b)] $\dim \ccHH^*(f,\widehat G) = \dim H^\ast(S_3)$ for $\widehat G \subseteq \widehat \SL_f$ if and only if $G \cong V_4$ is conjugate to a Klein four--group. 
    \end{itemize}
\end{theorem}

Proof of the theorem occupies Section~\ref{section: computations} along with some preparational propositions given in the preceding sections.

This theorem suggests that the mirror partner of $X_{\widetilde f, J}$ is $(f, \widehat V_4)$.
A very important consequence of this theorem is the following observation.

\begin{center}
 \textit{In mirror symmetry it's not enough to consider only the symmetry groups acting diagonally even for invertible polynomials. 
}
\end{center}

This is important to stress that the theorem above considers not only diagonal symmetry groups. It also covers some symmetry groups $G$ that are not abelian. Such symmetry groups are not widely considered from the point of view of Landau--Ginzburg orbifolds. The following list of citations is pretty much the full list of publications assuming nonabelian orbifolds \cite{Muk,BI21,BI22,BI24,WWP,CJMPW23,EGZ20,EGZ18}. And even in this list (except \cite{BI24}) one considers only the semidirect product groups $G^d \rtimes S$ with $G^d$ acting diagonally and $S$ being a subgroup of a symmetric group. To our knowledge, no investigation of a simple group was made before except \cite{BI24}.

Comparing our theorem to the mirror theorem of Efimov and Seidel it's important to note that their Landau--Ginzburg orbifold was suitably chosen, while our investigation is in the setup of the general approach of the Landau--Ginzburg orbifolds given by the invertible polynomials. The polynomial $f_g$ of Efimov and Seidel is not invertible whatever $g$ one takes.

\subsection{Acknowledgements}
The author is grateful to an anonymous referee for many useful comments that helped to improve the text and also find a mistake in the early version of this paper.

The work of Alexey Basalaev was supported by the Theoretical Physics and Mathematics Advancement Foundation "BASIS".

\section{Notations and terminologies}\label{sec:notations}
For a non-negative integer $N$ and a polynomial $f=f(x_1,\dots, x_N)\in\CC[x_1,\dots, x_N]$,
the {\em Jacobian algebra} $\Jac(f)$ of $f$ is a $\CC$-algebra defined as
\begin{equation}
\Jac(f)= \QR{\CC[x_1,\dots, x_N]}{\left(\frac{\p f}{\p x_1},\dots,\frac{\p f}{\p x_N}\right)}.
\end{equation}
The polynomial $f$ is said to define an isolated singularity if $\Jac(f)$ is a finite-dimensional $\CC$--vector space. In this case set $\mu_f:=\dim_\CC\Jac(f)$ and call it the {\em Milnor number} of $f$.
For $N=0$ set $\Jac(f):=\CC$ with $\mu_f=1$.

Let the {\em Hessian} of~$f$ be defined as the following determinant:
\begin{equation}
\hess(f):=\det \left(\frac{\partial^2 f}{\partial x_{i} \partial x_{j}}\right)_{i,j=1,\dots,N}.
\end{equation}

Its class is nonzero in $\Jac(f)$ giving the $\CC$--bilinear nondegenerate pairing $\eta_f$ called \textit{the residue pairing} (see \cite[Chapter 5]{GH94}, \cite[Section 5.11]{AGV85}).
The value $\eta_f([u],[v])$ is taken as the projection of the product $[u][v]$ to the $\CC$--span of $[\hess(f)]$. In particular, we fix
\[
	\eta_f([1],[\hess(f)]) = 1.
\]

\begin{remark}
One notes immediately that under the coordinate--wise action of $g \in \GL_f$ we have $g \cdot \hess(f) = (\det g)^2 \hess(f)$. This shows that the groups ${ G \subseteq \lbrace g \in \GL_f \ | \ \det g = \pm 1 \rbrace}$ preserve the pairing $\eta_f$. We will comment more on this in Proposition~\ref{prop: Jacf invariants}.
\end{remark}


\subsection{Symmetries}\label{sec: symmetries}

Given a quasihomogeneous polynomial $f = f(x_1,\dots,x_N)$ consider the
\textit{maximal group of linear symmetries} of $f$ defined by
\[
	\GL_f := \left\lbrace g \in \GL(N,\CC) \ | \ f(g \cdot \bx) = f(\bx) \right\rbrace.
\]

\begin{remark}
 This group contains the so--called maximal group of \textit{diagonal} symmetries of $f$ defined by
	$G_f^d := \left\lbrace (\lambda_1,\dots,\lambda_N) \in (\CC^*)^N \ | \ f(\lambda_1x_1,\dots,\lambda_Nx_N) = f(x_1,\dots,x_N) \right\rbrace.$
\end{remark}
%
%

%
For each $g\in \GL_f$, denote by $\Fix(g)$ \textit{the fixed locus of $g$}
\begin{equation}
 	\Fix(g):=\left\lbrace (x_1,\dots,x_N) \in\CC^N \ | \ g \cdot (x_1,\dots,x_N) = (x_1,\dots,x_N) \right\rbrace.
\end{equation}
This is an eigenvalue $1$ subspace of $g$ and therefore a linear subspace of $\CC^N$.
By $N_g:=\dim_\CC\Fix(g)$ denote its dimension and by $f^g:=f|_{\Fix(g)}$ the restriction of $f$ to the fixed locus of $g$.
For $g\in G_f^d$ this linear subspace is furthermore a span of a collection of standard basis vectors.
%
%
\begin{proposition}
	For any $g \in \GL_f$ with $N_g > 0$ there is a choice of coordinates on $\Fix(g)$ linear in $x_i$, such that the polynomial $f^g$ defines a quasihomogeneous singularity.
\end{proposition}
\begin{proof}
	Let $\widetilde x_1,\dots, \widetilde x_{N_g}, \widetilde x_{N_g+1},\dots,\widetilde x_N$ be the coordinates of $\CC^N$, such that $\widetilde x_1,\dots, \widetilde x_{N_g}$ are dual to the eigenspace $1$ basis vectors of $g$. In this coordinates we have $f^g = f \mid _{\widetilde x_{N_g+1} = \dots = \widetilde x_N = 0}$. The proof follows now by the same argument as in Proposition~5 of \cite{ET13}.
\end{proof}

Denote also
\[
\SL_f:=\GL_f\cap\SL(N,\CC).
\]
This group will be important later on because it preserves the volume form of $\CC^N$.

\section{Phase spaces}\label{sec: phase space}
This section is devoted to the definition of the state spaces and also the computational aspects of them. These are the A model state space $H^\ast(S_3)$ and the B model state space $\ccHH^*(f,G)$. 

\subsection{A--model total space}\label{sec: cohomology of Sg}
Recall that ${H^\ast_{orb}(S_g) = H^\ast(S_g)}$ because $S_g$ is a smooth manifold.
The dimension of $H^\ast(S_g)$ is $2+2g$. This is a $\ZZ/2\ZZ$--graded algebra. In particular, we have $H^\ast(S_g) = H^{even} \oplus H^{odd}$ with
\[
    H^{even} = \CC \cdot \delta \oplus \CC \cdot \gamma \quad H^{odd} = \bigoplus_{k=1}^g (\CC \cdot \alpha_k \oplus \CC \cdot \beta_k).
\]
Here $\delta \in H^0(S_g)$, $\gamma \in H^2(S_g)$ and $\alpha_k,\beta_k \in H^1(S_g)$. We can also assume the generators above to be such that the only non--zero products are
\[
    \alpha_k \circ \beta_k = -\beta_k \circ \alpha_k = \gamma, \quad \delta \circ a = a, \quad a \in H^\ast(S_g).
\]
Such an algebra has a unique up to a $\CC^\ast$--multiple Frobenius algebra structure, however this is not a subject of the current note.

\subsection{B--model total space}\label{sec: the total space}
For each $g\in \GL_f$ fix a generator of a one-dimensional vector space $\Lambda(g):=\bigwedge^{N -N_g}(\CC^N/\Fix(g))$. Denote it by $\xi_g$.


Define $\B_{tot}(f)$ to be the $\CC$--vector spaces of dimension $\sum_{g \in \GL_f} \dim \Jac(f^g)$
\begin{equation}\label{eq: Btot in jac sectors}
	\B_{tot}(f) := \bigoplus_{g \in \GL_f} \Jac(f^g) \xi_g,
\end{equation}
Each direct summand $\Jac(f^g) \xi_g$ will be called the \textit{$g$--th sector} and denoted by $\B_g'$.
Any $h\in\GL_f$ induces a map
$$h\colon \Fix(g)\to \Fix(hgh^{-1}).$$
and hence
$$h\colon \Lambda(g)\to \Lambda(hgh^{-1}).$$ Since we have fixed the generators $\xi_\bullet$, the latter map provides a constant $\rho_{h,g}\in\CC^*$ such that
\[
	h\left( \xi_{g} \right) = \rho_{h,g} \xi_{hgh^{-1}}.
\]
Then $\rho_{h_2,h_1gh_1^{-1}}\rho_{h_1,g}=\rho_{h_2h_1,g}$. Define the action of $\GL_f$ on $\B_{tot}$ by
\begin{align}
h^*([p(\bx)]\xi_{g})=\rho_{h,g}[p(h^{-1}\cdot\bx)]\xi_{hgh^{-1}}.
\end{align}

For a finite $G \subseteq \GL_f$ put
\[
\B_{tot,G} := \bigoplus_{g \in G} \B'_g = \bigoplus_{g \in G} \Jac(f^g) \xi_g\subseteq \B_{tot}
\]
and define as a vector space
\begin{equation}
	\ccHH^*(f,G) := \left( \B_{tot,G} \right)^G.
\end{equation}
Namely, the linear span of the $\B_{tot}$ vectors that are invariant with respect to the action of all elements of $G$.

The latter direct sum and the group action can be intertwined as follows. Let $\C^G$ stand for the set of representatives of the conjugacy classes of $G$ and $Z(g)$ for the centralizers of $g \in G$. Denote $\B_g := \left( \B_g' \right)^{Z(g)}$. Then we have
\begin{equation}\label{eq: Bfg via centralizers}
    \ccHH^*(f,G) \cong \bigoplus_{g \in \C^G} \B_g = \bigoplus_{g \in \C^G} \left( \B_g' \right)^{Z(g)}.
\end{equation}
cf. \cite[Proposition 42]{BI22} for the proof.

It's much more complicated to define the algebra structure of $\ccHH^*(f,G)$ and we don't do it here because it will not be needed later on. We will only need that the product $\circ$ of $\ccHH^*(f,G)$ respects the $G$--grading: $\circ: \B_g \otimes \B_h \to \B_{gh}$, its unit is $1 \xi_\id \in \B_\id$ and $\circ: \B_\id \otimes \B_\id \to \B_\id$ coincides with the product of $(\Jac(f))^G$.

\subsection{Computational statements}\label{sec: calculus-1}
The following proposition follows immediately from the definition but will be very important later on.

\begin{proposition}\label{prop: g--action on itself}
    Let $f$ define an isolated singularity. Then we have
    \begin{description}
     \item[(i)] Let $g \in \GL_f$ and $k \in \ZZ$ be such that $\Fix(g^k) = \Fix(g)$ then $g^\ast \left( \xi_{g^k} \right) = \det g \cdot \xi_{g^k}$.
     \item[(ii)] Let $G \subseteq \GL_f$ be abelian. Then $\ccHH^*(f,G) = \bigoplus_{g \in G} \left( \Jac(f^g) \xi_g \right)^G $.
    \end{description}
    \
\end{proposition}
\begin{proof}
    Let $\widetilde x_1,\dots,\widetilde x_{N_g}$ be the coordinates dual to the eigenvalue $1$ subspace of $g$. Then $(d\widetilde x_1 \wedge \dots \wedge d \widetilde x_{N_g} ) \wedge \xi_{g^k}$ is a volume form of $\CC^N$. We get
    \begin{align*}
        & \det g \cdot (d\widetilde x_1 \wedge \dots \wedge d \widetilde x_{N_g} ) \wedge \xi_{g^k} = g \left((d\widetilde x_1 \wedge \dots \wedge d \widetilde x_{N_g} ) \wedge \xi_{g^k} \right)
        \\
        & \quad = \rho_{g,g^k} \cdot (d\widetilde x_1 \wedge \dots \wedge d \widetilde x_{N_g} ) \wedge \xi_{g^k}.
    \end{align*}
    This gives the first claim.
       
    In order to show the second claim consider Eq.~\eqref{eq: Bfg via centralizers}. We have $\C^G = G$ and $Z(g) = G$ for any $g \in G$, what completes the proof.
\end{proof}

\begin{proposition}
    Let $G_1,G_2 \subseteq \GL_f$ be conjugate. Then $\ccHH^*(f,G_1) \cong \ccHH^*(f,G_2)$.
\end{proposition}
\begin{proof}
    Let $G_2 = \widetilde g G_1 \widetilde g^{-1}$ for some $\widetilde g \in \GL(\CC^N)$.
    Consider the decomposition of Eq.\eqref{eq: Bfg via centralizers} for both groups. We show that
    \[
        \Psi: [\phi(x)] \xi_h \mapsto [\phi(g \cdot x)] \xi_{\widetilde g h \widetilde g^{-1}}, \quad \forall h \in G_1
    \]
    establishes the isomorphism $\ccHH^*(f,G_1) \xrightarrow{\sim} \ccHH^*(f,G_2)$.

    Obviously $\C^{G_2} = \lbrace \widetilde g r \widetilde g^{-1} \ | \ r \in \C^{G_1} \rbrace$ and $Z(\widetilde g h \widetilde g^{-1}) = \widetilde g Z(h) \widetilde g^{-1}$.
    Note that we have the isomorphism $\Fix(h) \xrightarrow{\sim} \Fix( \widetilde g h \widetilde g^{-1})$. This gives $\B'_h = \B'_{\widetilde g h \widetilde g^{-1}}$.
    
    In order to finish the proof it remains to show that the invariants under the action of $Z(h)$ and $Z(\widetilde g  h \widetilde g^{-1} )$ agree.
    Let $r = \widetilde g s \widetilde g^{-1} \in \widetilde g Z(h) \widetilde g^{-1}$. Then
    we have
    \[
     r^\ast \xi_{\widetilde g h \widetilde g^{-1}} = \rho \cdot \xi_{\widetilde g s \widetilde g^{-1}}
     \quad
     \text{and}
     \quad
     s^\ast \xi_{\widetilde g h \widetilde g^{-1}} = \rho' \cdot \xi_{s}
    \]
    for some $\rho,\rho' \in \CC^\ast$. The following computation shows that $\rho = \rho'$.
    \[
        r^\ast \xi_{\widetilde g h \widetilde g^{-1}} = (\widetilde g s \widetilde g^{-1})^\ast \xi_{\widetilde g h \widetilde g^{-1}} = s ^\ast \xi_{h}
    \]
    where the last equality follows from the fact that $s$ acts on $\Lambda(h)$ by the same multiple as $\widetilde g s \widetilde g^{-1}$ acts on $\Lambda (\widetilde g h \widetilde g^{-1})$.
\end{proof}

The following proposition shows that the pairing restricts to the invariants of ${G \subseteq \widehat\SL_f}$.

\begin{proposition}\label{prop: Jacf invariants}
    Let $G \subset \GL_f$ be such that $\det g = \pm 1$ for any $g \in G$. Consider the polynomials $\phi_1(\bx),\phi_2(\bx) \in \CC[\bx]$ satisfying $\eta_f([\phi_1],[\phi_2]) \neq 0$.

    If $[\phi_1]\neq 0$ in $(\Jac(f))^G$, then $[\phi_2]\neq 0$ in $(\Jac(f))^G$ as well.
\end{proposition}
\begin{proof}
    Any $g \in \GL_f$ acts on the hessian polynomial by $g \left( \hess(f) \right) = (\det g)^2 \hess(f)$. Let $\eta_f([\phi_1],[\phi_2]) = \alpha \in \CC^\ast$. The following equality holds in $\CC[\bx]$ for some polynomials $p_1,\dots,p_N$
    \[
        \phi_1(\bx) \cdot \phi_2(\bx) - \alpha \hess(f) = \sum_{k=1}^N p_k(\bx) \frac{\p f}{\p x_k}.
    \]
    The Jacobian ideal of $f$ is preserved under the $G$--action. Hence under the assumption of the proposition we have
    \[
        [\phi_1(\bx)] \cdot [\phi_2(\bx)] =  \alpha [\hess(f)] \neq 0 \in (\Jac(f))^G
    \]
    what completes the proof.

\end{proof}


\subsection{$\ccHH^*(f,G)$ for a Klein quartic}\label{sec: calculus-2}
The following propositions are specific for $f$ being the Klein quartic polynomial.
We consider in details the cases of $G \subseteq \GL_f$ with $\det g = 1$ and $(\det g)^2 = 1$.

\subsubsection{Consider $g \in G \backslash \lbrace \id \rbrace$ with $\det g = 1$}\label{section: det 1 g in klein}
The fixed locus of $g$ corresponds to the eigenvalue $1$ subspace and could be either $0$ or $1$--dimensional. If $\Fix(g) = 0$ the element $\xi_g$ is a non--zero multiple of a volume form on $\CC^N$ (see Section~\ref{sec: the total space}). Then any element $h \in Z(g)$ acts on $\xi_g$ by multiplication by $\det h$. In particular, $\xi_g$ is $G$--invartiant if $G \subset \SL(3,\CC)$ and is not $G$--invariant if $-\id \in G$.

If $\Fix(g) \cong \CC$, we have $f^g = \widetilde x^4$ where $\widetilde x$ is a coordinate of $\Fix (g)$. This gives $\B'_g \cong \CC \langle [1], [\widetilde x], [\widetilde x^2] \rangle \xi_g$. Let $h \in Z(g)$ act on $\widetilde x$ by $\rho \in \CC^\ast$. Namely $h(\widetilde x) = \rho \widetilde x$. Then we have $\rho_{h,g} = \det h / \rho$ by the same argument as in proof of Proposition~\ref{prop: g--action on itself}. This gives
\begin{equation}\label{eq: multiples of the centralizers}
h^\ast ( [\widetilde x ^k] \xi_g) = \frac{\det h}{\rho^{k+1}} \cdot [\widetilde x ^k] \xi_g. 
\end{equation}
In order to compute the direct sum contribution of $g$ to $\ccHH^*(f,G)$ it remains to check the multiples $\rho$ of all the centralizers $Z(g)$.

\subsubsection{Consider $g \in G$ with $\det g = -1$}\label{sec: minus one det two dim fix}
The fixed locus of $g$ can be either $0$, $1$--dimensional or $2$--dimensional. The first two cases are treated exactly in the same way as above and we concentrate on the last one. Let $\widetilde x_1, \widetilde x_2$ be the coordinates on $\Fix(g)$. Then $f^g$ is quasihomogeneous with weights $(1/4,1/4)$. After some rescaling of the variables $f^g$ can either be $\widetilde x_1^4 + \widetilde x_2^4$, $\widetilde x_1^3 \widetilde x_2 + \widetilde x_2^4$ or $\widetilde x_1^3 \widetilde x_2 + \widetilde x_2^3 \widetilde x_1$.

We have $\B'_g \cong \CC^9$ generated by the classes $[\widetilde x_1^{a_1} \widetilde x_2^{a_2}]$ for some pairs $(a_1,a_2)$. Since $\det g = -1$, we have $g^\ast (\xi_g) = - \xi_g$ and $g^\ast ([\widetilde x_1^{a_1} \widetilde x_2^{a_2}] \xi_g) = - [\widetilde x_1^{a_1} \widetilde x_2^{a_2}]\xi_g$. 
This shows that $(\B'_g)^G = 0$.

\subsubsection{Consider $g = \id$}
While computing the vector space $\ccHH^*(f,G)$ one has to find the $G$--invariants of $\Jac(f) \xi_\id$. We have $h^\ast (\xi_\id) = \xi_\id$ and it remains to compute the $G$--invariants of $\Jac(f)$ itself. This goal can be achieved by symmetrization procedure. Namely, if $\lbrace \phi_k(\bx) \rbrace_{k=1}^\mu$ are the monomials whose classes generate $\Jac(f)$, then the classes of $\lbrace \sum_{g \in G} \phi_k(g \cdot \bx) \rbrace_{k=1}^\mu$ generate $(\Jac(f))^G$. However, some of these classes might be zero or linearly dependant.

The basis monomials of $\phi_\bullet$ can be taken to be $x_1^{a_1}x_2^{a_2}x_3^{a_3}$, such that $\sum_k a_k \le 6$ and $0 \le a_k \le 2$. Also $[\hess(f)] = 756 [(x_1x_2x_3)^2]$.
Assign the grading $\sum_k a_k$ to each of these monomials. The vector space $\Jac(f)$ decomposes into the direct sum of the graded pieces. The dimensions of these pieces can be composed into the following vector
\[
    1,\ 3, \ 6, \ 7, \ 6, \ 3, \ 1.
\]
This vector is symmetric by its middle point --- the $4$--th component because the pairing $\eta_f$ is non--degenerate.

The grading introduced is preserved under the coordinate-wise action of $g \in \GL_f$. Similarly $(\Jac(f))^G$ decomposes into the direct sum of the graded pieces and the dimensions of these pieces can be written by a vector. This vector will be symmetric by its middle point due to Proposition~\ref{prop: Jacf invariants}.

\section{Admissible symmetry groups}\label{section: computations}

This section gathers the computations of $\ccHH^*(f,G)$ for a Klein quartic polynomial $f$ and all admissible symmetry groups $G \subset \GL_f$. Namely, the groups $G$, such that $\det g = \pm 1$ for any $g \in G$.

\subsection{The subgroups}

The group $\SL(f)$ is generated by the following $\SL(3,\CC)$ elements
\begin{align*}
R=\frac{\sqrt{-1}}{\sqrt{7}} \left(
\begin{array}{ccc}
 \zeta -\zeta ^6 & \zeta ^2-\zeta ^5 & \zeta ^4-\zeta ^3 \\
 \zeta ^2-\zeta ^5 & \zeta ^4-\zeta ^3 & \zeta -\zeta ^6 \\
 \zeta ^4-\zeta ^3 & \zeta -\zeta ^6 & \zeta ^2-\zeta ^5 \\
\end{array}
\right),
\
 T=\left(
\begin{array}{ccc}
 0 & 1 & 0 \\
 0 & 0 & 1 \\
 1 & 0 & 0 \\
\end{array}
\right),
\
S=\left(
\begin{array}{ccc}
 \zeta ^4 & 0 & 0 \\
 0 & \zeta ^2 & 0 \\
 0 & 0 & \zeta  \\
\end{array}
\right).
\end{align*}
where $\zeta = \exp\left(2 \pi \sqrt{-1}/7\right)$. This was observed already by F.Klein~\cite{K79}.
The following list gives the classification of the subgroups of $\SL(f)$ up to conjugation

\begin{enumerate}
 \item[(a)] 8 conjugate elementary abelian groups of order 7,
 \item[(b)] 28 conjugate cyclic groups of order 3,
 \item[(c)] 21 conjugate cyclic groups of order 4,
 \item[(d)] 21 conjugate cyclic groups of order 2,
 \item[(e)] two classes of 7 conjugate dihedral abelian Klein 4-groups of order 4,
 \item[(f)] 28 dihedral nonabelian groups of order 6,
 \item[(g)] 21 dihedral nonabelian groups of order 8,
 \item[(h)] 8 nonabelian groups of order 21,
 \item[(i)] two classes of 7 nonabelian conjugates of the symmetric group of degree 4,
 \item[(j)] two classes of 7 nonabelian conjugates of the alternating group of degree 4.
\end{enumerate}

For any group $G$ from the list above we also consider its extension by $\ZZ/2\ZZ$ obtained by adding an element $-\id \in \GL(3,\CC)$. Denote such group by $\widehat G$. Obviously, if $g_1,\dots,g_r$ are all the representatives of the conjugacy classes of $G$, then $\pm g_1,\dots, \pm g_r$ are the representatives of the conjugacy classes of $\widehat G$.

\subsection{Presentation of the results}

In the following computations we make use of different statements of Section~\ref{sec: phase space} without quoting them each time. In particular, we make use of Eq.~\eqref{eq: Bfg via centralizers} to find $\ccHH^*(f,G)$ and Section~\ref{sec: calculus-1},\ref{sec: calculus-2} to compute the $G$--invariants. 
In particular, for any $G$ beneath we find the basis of $\ccHH^*(f,G)$.

In order to apply Eq.~\eqref{eq: Bfg via centralizers} we compute $\C^G$ and the centralizers $Z(g)$ for any representative $g \in \C^G$. Following Section~\ref{section: det 1 g in klein} for $g$, such that $\det(g) = 1$ we need to compute the \textit{multiples of the centralizers} as in Eq.~\eqref{eq: multiples of the centralizers} for all $h \in Z(g)$. We will only list the different values of these multiples in what follows.

The identity sector is often of the highest dimension. We compute its polynomial basis every time. Note that this basis can always be chosen to be homogeneous in the polynomial degree. We provide what we call the \textit{dimension vector} --- the length $7$ vector consisting of the number of the degree $0,\dots,6$ basis polynomials.

For the $\ZZ/2\ZZ$ extended groups $\widehat G$ only the computations of $\widehat V_4$ and $\widehat \SL_f$ are given explicitly. The computations of all the other groups are derived easily by using the statements of Section~\ref{sec: phase space}.

\subsection{The group $\SL_f$}
Let $G := \langle R,T,S \rangle$ be the full group $\SL_f$. It is well-known to have 6 conjugacy classes of orders $1, 24, 24, 21, 42, 56$. We have
\[
    \C^G = \lbrace \id, S, S^3, R, RS^3, T \rbrace.
\]
The centralizers are $Z(S) = Z(S^3) = \langle S \rangle$, $Z(T) = \langle T \rangle$, $Z(R) = \langle TRSRS^5, TRS^3RS^6 \rangle$ and $Z(RS^3) = \langle RS^3 \rangle$.

The elements $S$ and $S^3$ don't have $1$ as an eigenvalue, while all other do. The multiples of the centralizers are $\pm 1$ for $R$, and all equal to $1$ for $T$ and $RS^3$. This gives
\begin{align}
    & \B_{S} = \CC \langle \xi_{S} \rangle, \ \B_{S^3} = \CC \langle \xi_{S^3} \rangle, \ \B_R = \CC \langle [\widetilde x] \xi_R \rangle
    \\
    & \B_{RS^3} = \CC\langle \xi_{RS^3}, [\widetilde x] \xi_{RS^3}, [\widetilde x^2]\xi_{RS^3} \rangle, \quad
    \B_T = \CC\langle \xi_T, [\widetilde x] \xi_T, [\widetilde x^2]\xi_T \rangle
\end{align}

The identity sector $\B_\id$ is spanned by
\[
    1 \xi_\id \quad \text{and} \quad [5 x_2^2 x_3^2 x_1^2 -x_3 x_1^5-x_1 x_2^5-x_2 x_3^5]\xi_\id.
\]
Both classes are $G$--invariant and non--zero in $\Jac(f)$.

$\ccHH^*(f,G)$ has $2 + 9 = 11$ basis elements in total.

\subsection{The group (a)}
Let $G$ be generated by $S$. This is an order $7$ group and $\Fix(g) = 0$ for any nonidentical element. This gives $\B_g = \CC \cdot \xi_g$ for any such $g$.

In the identity sector the space $\B_\id$ is spanned by the elements
\begin{align*}
    \xi_\id, \quad [x_1x_2x_3] \xi_\id, \quad [(x_1x_2x_3)^2] \xi_\id.
\end{align*}

$\ccHH^*(f,G)$ has $3 + 6 = 9$ basis elements in total.

\subsection{The group (b)}
Let $G$ be generated by $T$. We have $\C^G = \lbrace \id, T, T^2\rbrace$ and 1 is an eigenvalue of $T,T^2$. This gives
\[
    \B_g = \CC \langle \xi_g, [(x_1+x_2+x_3)] \xi_g, [(x_1+x_2+x_3)^2] \xi_g \rangle,
    \quad g = T,T^2.
\]

The identity sector basis is given by $11$ elements $[\phi(x)] \xi_\id$ with $\phi(x)$ taken from the list below
\begin{align*}
    & 1, \quad x_1+x_2+x_3, \quad x_1^2+x_2^2+x_3^2, x_2 x_3+x_1x_2+x_1x_3, \quad x_3 x_1^2 + x_2^2 x_1 +x_2 x_3^2, x_1 x_2 x_3,
    \\
    & x_2 x_1^2+x_3^2 x_1+x_2^2 x_3, \quad x_1 x_2 x_3 \left(x_1+x_2+x_3\right), \ x_1^2x_2^2+x_1^2x_3^2 +x_2^2 x_3^2,
    \\
    & x_1 x_2 x_3 \left(x_2 x_3+x_1 x_2+x_1x_3\right), \quad x_1^2 x_2^2 x_3^2.
\end{align*}
The corresponding dimension vector reads
\[
    1, \ 1, \ 2, \ 3, \ 2, \ 1, \ 1.
\]

$\ccHH^*(f,G)$ has  $11 + 2\cdot 3 = 17$ basis elements in total.

\subsection{The group (c)}
Let $G$ be generated by $g := RSRS^5$. The generator has $1$ as an eigenvalue, $\C^G = G$ and 
\[
    \B_{g^k} \cong \CC \langle [1] \xi_{g^k}, [\widetilde x ] \xi_{g^k}, [\widetilde x^2 ] \xi_{g^k}\rangle , \quad k =1,2,3.
\]

The identity sector basis is given by $9$ elements $[\phi(x)] \xi_\id$ with the dimension vector
\[
    1, \ 1, \ 2, \ 1, \ 2, \ 1, \ 1.
\]
The polynomials $\phi(x)$ can be taken to be the symmetrizations of
\[
    1, \quad x_1, \quad x_3^2, \ x_2 x_3, \quad x_2 x_3^2, \quad x_1^2 x_2^2, \ x_1^2 x_2 x_3, \quad x_1(x_2x_3)^2, \quad (x_1x_2x_3)^2.
\]
$\ccHH^*(f,G)$ has  $9 + 3\cdot 3 = 18$ basis elements in total.

\subsection{The group (d)}
Let $G$ be generated by $g = RT$. The generator has $1$ as an eigenvalue.
We have $\C^G = G$ and $\B_g \cong \CC^3$.

The identity sector basis is given by $15$ elements $[\phi(x)] \xi_\id$ with the dimension vector
\[
    1, \ 1, \ 4, \ 3, \ 4, \ 1, \ 1.
\]
The polynomials $\phi(x)$ can be taken to be the symmetrizations of
\[
    1, \ x_1, \ x_3^2, x_2 x_3, x_2^2, x_1 x_3, \ x_1 x_2 x_3, x_1 x_3^2, x_2 x_3^2, \ x_1^2 x_2^2, x_1^2 x_2 x_3, x_1^2 x_3^2, x_1x_2^2x_3, \ x_1(x_2x_3)^2, \ (x_1x_2x_3)^2.
\]

$\ccHH^*(f,G)$ has  $15 + 3 = 18$ basis elements in total.

\subsection{The group (e) - Klein 4-group}
Let $G := \{\id,RS^2RS,SRS^6,S^2RS^3RS \}$. It's not hard to see that $G$ is indeed homomorphic to a Klein four-group.  All nontrivial group elements have eigenvalues $(1,-1,-1)$.

We have $\C^G = G$ and the centralizers being all $G$. For the nonidentical elements $g$, the multiples of the centralizers are $\pm 1$ and therefore
\[
    \B_g = \CC \langle [\widetilde x] \rangle \xi_g \quad \forall g \in G \backslash \{\id \}.
\]

The identity sector basis is given by $9$ elements $[\phi(x)] \xi_\id$ with the dimension vector
\[
    1, \ 0, \ 3, \ 1, \ 3, \ 0, \ 1.
\]
The polynomials $\phi(x)$ can be taken to be the symmetrizations of
\[
    1, \quad x_3^2, \ x_2 x_3, \ x_2^2, \quad x_2 x_3^2, \quad x_1^2 x_2^2, \ x_1^2 x_2 x_3, \ x_1^2 x_3^2, \quad (x_1x_2x_3)^2.
\]

$\ccHH^*(f,G)$ has  $9 + 3 = 12$ basis elements in total.

\subsection{The group (f) - $D_6$}
Let $G$ be generated by $a = T$ and $b = R$.

We have $\C^G = \{\id, a, b\}$. All these elements have $1$ as an eigenvalue. The corresponding centralizers are $Z(a) = \langle a \rangle$, $Z(b) = \langle b\rangle$.

The multiples of the centralizers are all equal to $1$ in both cases. Due to this we have
\[
 \B_{a} = \CC \langle \xi_{a},[\widetilde x] \xi_{a},[\widetilde x^2] \xi_{a} \rangle, \quad
 \B_{b} = \CC \langle \xi_{b}, [\widetilde x] \xi_{b}, [\widetilde x^2] \xi_{b} \rangle.
\]

The identity sector basis is given by $7$ elements $[\phi(x)] \xi_\id$ with the dimension vector
\[
    1, \ 0, \ 2, \ 1, \ 2, \ 0, \ 1.
\]
The polynomials $\phi(x)$ can be taken to be the symmetrizations of
\[
    1, \quad x_3^2, \ x_2 x_3, \quad x_1x_2x_3, \quad x_1^2 x_2^2, \ x_1^2 x_2 x_3, \quad (x_1x_2x_3)^2.
\]

$\ccHH^*(f,G)$ has  $7 + 2\cdot 3= 13$ basis elements in total.

\subsection{The group (g) - $D_8$}
Let $G$ be generated by $a = RS^3$ and $b = RS^2RS$.

We have $\C^G = \{\id, a^2, b, ab, a\}$. All these elements have $1$ as en eigenvalue. The corresponding centralizers are $Z(a^2) = G$, $Z(b) = \langle a^2,b\rangle$, $Z(ab) = \langle a^2,ab\rangle$, $Z(a) = \langle a \rangle$.

The multiples of the centralizers are $\pm 1$ for $a^2$, $b$ and $ab$. The multiples of the centralizers for $a$ are $\pm \sqrt{-1},1$. Due to this we have
\[
 \B_{a^2} = \CC \langle [\widetilde x] \xi_{a^2} \rangle, \ \B_{b} = \CC \langle [\widetilde x] \xi_{b} \rangle, \ \B_{ab} = \CC \langle [\widetilde x] \xi_{ab} \rangle, \ \B_{a} = 0.
\]

The identity sector basis is given by $6$ elements $[\phi(x)] \xi_\id$ with the dimension vector
\[
    1, \ 0, \ 2, \ 0, \ 2, \ 0, \ 1.
\]
The polynomials $\phi(x)$ can be taken to be the symmetrizations of
\[
    1, \quad x_3^2, \ x_2 x_3, \quad x_1^2 x_2^2, \ x_1^2 x_2 x_3, \quad (x_1x_2x_3)^2.
\]

$\ccHH^*(f,G)$ has  $6 + 3 = 9$ basis elements in total.

\subsection{The group (h)}
Let $G$ be generated by $T$ and $S$.

We have $\C^G = \{\id, T, T^2, S, S^2\}$. Both $S$ and $S^2$ do not have $1$ as an eigenvalue and therefore give one--dimensional sectors. The elements $T$ and $T^2$ both have $1$ as an eigenvalue ans the centralizers in both cases are $\langle T \rangle \subset G$. We have
\[
    \B_{S} = \CC \cdot \xi_S, \ \B_{S^2} = \CC \cdot \xi_{S^2}, \
    \B_{T} = \CC \langle \xi_T,[\widetilde x] \xi_T,[\widetilde x^2] \xi_T \rangle, \
    \B_{T^2} = \CC \langle \xi_{T^2},[\widetilde x] \xi_{T^2},[\widetilde x^2] \xi_{T^2} \rangle
\]

The identity sector basis is given by $3$ elements $[\phi(x)] \xi_\id$ with the dimension vector
\[
    1, \ 0, \ 0, \ 1, \ 0, \ 0, \ 1.
\]
The polynomials $\phi(x)$ can be taken to be
\[
    1, \ x_1 x_2 x_3, \ (x_1 x_2 x_3)^2.
\]

$\ccHH^*(f,G)$ has  $3 + 8 = 11$ basis elements in total.

\subsection{The group (i)}
Consider
\[
 c_2 := TS^5RS^6, \quad c_3 := TS^4 \quad c_4 := TRS^2RS^3, \quad v_4 := T^2RS^6RS^4.
\]
These elements have orders $2,3,4,2$ respectively. The first three represent 2--cycle, 3--cycle and 4--cycle in $S_4$ embedded in $\SL_f$ as a group $G := \langle c_2,c_4 \rangle$. The element $v_4$ above is the Klein 4--group element of $G$.

We have $\C^G = \{\id, c_2, c_3, c_4, v_4\}$. All the elements listed have $1$ as an eigenvalue.
The centralizers are $Z(c_2) = \lbrace \id, RSRS^2, RS^4RS^4,RS^5RS^6\rbrace$ giving the multiples $\pm 1$;
$Z(c_3) = \lbrace \id, TS^4,T^2S^5\rbrace$ giving the multiples $1$;
$Z(c_4) = \lbrace \id, TRSRS^5, TRS^2RS^3, RS^6RS \rbrace$ giving the multiples $1$;
$Z(v_4) = \{ \id, RS^4RS^3, R,T^2RS^2RS^6, T^2RS^3RS, T^2RS^5RS^2, T^2RS^6RS^4, TRSRS^5 \}$ giving the multiples $\pm 1$.
We have
\[
    \B_{c_3} = \CC \langle \xi_{c_3}, [\widetilde x_1] \xi_{c_3}, [\widetilde x_1^2] \xi_{c_3} \rangle, \quad
    \B_{c_4} = \CC \langle \xi_{c_4}, [\widetilde x_1] \xi_{c_4}, [\widetilde x_1^2] \xi_{c_4} \rangle, \quad
    \B_{c_2} = \CC \langle [\widetilde x_1] \xi_{c_2} \rangle, \quad
    \B_{v_4} = \CC \langle [\widetilde x_1] \xi_{v_4} \rangle.
\]

The identity sector basis is given by $4$ elements $[\phi(x)] \xi_\id$ with the dimension vector
\[
    1, \ 0, \ 1, \ 0, \ 1, \ 0, \ 1.
\]
The polynomials $\phi(x)$ can be taken to be
\[
    1, \ x_3^2, \ (x_1 x_2)^2, \ (x_1 x_2 x_3)^2.
\]

$\ccHH^*(f,G)$ has  $4 + 8 = 12$ basis elements in total.

\subsection{The group (j)}
Consider the subgroup $A_4 \subset S_4$ for $S_4$ emebedded in $\SL_f$ as in the case above. Let the notation be as above.

We have $\C^G = \{\id, c_3, v_4\}$. All the elements listed have $1$ as an eigenvalue.
The centralizers are 
$Z(v_4) = \lbrace \id, T^2RS^3RS, T^2RS^5RS^2, T^2RS^6RS^4\rbrace $ giving the multiples $\pm 1$;
$Z(c_3) = \lbrace \id, TS^4,T^2S^5\rbrace$ giving the multiple $1$ exclusively.
We have

\[
    \B_{c_4} = \CC \langle \xi_{c_4}, [\widetilde x_1] \xi_{c_4}, [\widetilde x_1^2] \xi_{c_4} \rangle, \quad
    \B_{v_4} = \CC \langle [\widetilde x_1] \xi_{v_4} \rangle.
\]

The identity sector basis is given by $10$ elements $[\phi(x)] \xi_\id$ with the dimension vector
\[
    1, \ 0, \ 3, \ 2, \ 3, \ 0, \ 1.
\]
The polynomials $\phi(x)$ can be taken to be
\[
    1, \quad x_3^2, \ x_2 x_3, \ x_2^2, \quad x_1 x_3^2, \ x_2 x_3^2, \quad (x_1x_2)^2, \ x_1^2x_2x_3, \ (x_1x_3)^2, \quad (x_1 x_2 x_3)^2.
\]

$\ccHH^*(f,G)$ has  $10 + 4 = 14$ basis elements in total.

\subsection{$\ZZ/2\ZZ$--extension of $\SL_f$}
In this section denote $\widehat G := \lbrace \pm g  \ | \ g \in G\rbrace$ with $g \in \SL_f$.
\[
    \C^{\widehat G} = \lbrace \pm\id, \pm S, \pm S^3, \pm R, \pm RS^3, \pm T \rbrace.
\]

We have $\B_{\pm S} = \B_{\pm S^3} = \B_{-\id} = 0$ because the fixed locus of these elements is $0$ and the respective generator $\xi_g$ is not preserved by the action of $-\id$.

The fixed loci of $R$ and $-R$ are $1$ and $2$--dimensional respectively. 
\[
\B_R \cong \left( \CC \langle [\widetilde x] \xi_R \rangle \right)^{\langle -\id \rangle} = 0
\] 
and $\B_{-R} = 0$ as explained in Section~\ref{sec: minus one det two dim fix}.

For $g \in \lbrace \pm RS^3, \pm T \rbrace$ the fixed locus is $1$--dimensional giving
\begin{align}
    \B_{g} \cong \left( \CC \langle \xi_{g}, [\widetilde x] \xi_{g}, [\widetilde x^2] \xi_{g} \rangle \right)^{\langle -\id \rangle}
    \cong \CC \langle \xi_{Rg}, [\widetilde x^2] \xi_{g} \rangle.
\end{align}

The identity sector is spanned by
\[
    1 \xi_\id \quad \text{and} \quad [5 x_2^2 x_3^2 x_1^2 -x_3 x_1^5-x_1 x_2^5-x_2 x_3^5]\xi_\id.
\]

We conclude that $\ccHH^*(f,\widehat G)$ is of dimension $2\cdot 4 + 2 = 10$.

\subsection{$\ZZ/2\ZZ$--extension of the Klein 4-group}
In this section denote $\widehat G := \lbrace \pm g  \ | \ g \in G\rbrace$ with $G := \{\id,RS^2RS,SRS^6,S^2RS^3RS \}$.

Then for any $g \in G$ the eigenvalues of $g$ are $(1,-1,-1)$ and those of $-g$ are $(-1,1,1)$. We get $\Fix(-g) \cong \CC^2$, however the sector $\B_{-g} = 0$ as explained in Section~\ref{sec: minus one det two dim fix}. Similarly $\B_g = ( \CC \langle [\widehat x] \rangle \xi_g )^{\widehat G} =0$ because $[\widehat x] \xi_g$ is not invariant under the action of $-\id$.

We conclude that
\[
    \ccHH^*(f, \widehat G) = \left( \Jac(f) \xi_\id \right)^{\widehat G} \cong \left( \Jac(f) \right)^{\widehat G} \cdot \xi_\id.
\]
By using the computations done for the group $G$ we have
\[
    \ccHH^*(f,\widehat G) \cong \CC\langle [1], [x_3^2],[x_2 x_3],[x_2^2], [x_1^2 x_2^2],[x_1^2 x_2 x_3],[x_1^2 x_3^2], [(x_1x_2x_3)^2]\rangle \xi_\id.
\]

We conclude that as a $\CC$--vector space $\ccHH^*(f,\widehat G) \cong \CC^{2 + 2\cdot 3} \cong H^\ast(S_3)$ --- the cohomology of the genus $3$ Riemann surface. However this cannot be an algebra isomorphism because $H^\ast(S_3)$ is not a commutative algebra while the product of $\ccHH^*(f,\widehat G)$ above is commutative.
It is $8$--dimensional with the product restricted from $\Jac(f)$. The only non--zero products that do not involve the element $1 \xi_\id$ are
\begin{align*}
    & [x_2^2]\xi_\id \circ [x_1^2 x_3^2] \xi_\id = [(x_1x_2x_3)^2] \xi_\id,
    \quad
    [x_3^2]\xi_\id \circ [x_1^2 x_2^2] \xi_\id = [(x_1x_2x_3)^2] \xi_\id
    \\
    & \qquad [x_2 x_3] \xi_\id \circ [x_1^2 x_2 x_3] \xi_\id = [(x_1x_2x_3)^2] \xi_\id.
\end{align*}
This suggests that the product structure for the B--model state space might be defined differenly for the algebra isomorphism to hold true.


\end{document}